\documentclass[11pt]{amsart}
\usepackage{latexsym,amssymb,amsmath,amscd}

\newcommand{\rar}{\rightarrow}
\newcommand{\lar}{\longrightarrow}

\newtheorem{Theorem}{Theorem}[section]

\newtheorem{Corollary}[Theorem]{Corollary}
\newtheorem{Proposition}[Theorem]{Proposition}
\newtheorem{Remark}[Theorem]{Remark}
\newtheorem{Example}[Theorem]{Example}
\newtheorem{Definition}[Theorem]{Definition}
\newtheorem{Question}[Theorem]{Question}

\def\ann{\mbox{\rm ann}}
\def\Sym{\mbox{\rm Sym}}
\def\height{\mbox{\rm height }}

\def\AA{{\mathbf A}}
\def\BB{{\mathbf B}}
\def\RR{{\mathbf R}}
\def\TT{{\mathbf T}}
\def\ttt{{\mathbf t}}
\def\H{{\mathrm H}}
\def\m{{\mathfrak m}}

\begin{document}

\title[On complete monomial ideals]{\sc
 On  complete monomial ideals
}

\author{Philippe Gimenez}
\thanks{The first author was partially supported by {\it Ministerio de Econom\'{\i}a y Competitividad} (Spain), MTM2010-20279-C02-02.}
\address{Departamento de \'Algebra, An\'alisis Matem\'atico, Geometr\'{\i}a y Topolog\'{\i}a \\
Universidad de Valladolid \\
Facultad de Ciencias, , 47011 Valladolid, Spain}
\email[Corresponding author]{pgimenez@agt.uva.es}

\author{Aron Simis}
\thanks{The second author is partially supported by
a CNPq grant and a CAPES Senior Visiting Fellow Scholarship at the Universidade Federal da Paraiba (Brazil)}
\address{Departamento de Matem\'atica \\
Universidade Federal de Pernambuco \\
50740-560 Recife, PE, Brazil }
\email{aron@dmat.ufpe.br}

\author{Wolmer V. Vasconcelos}
\address{Department of Mathematics \\ Rutgers University \\
110 Frelinghuysen Rd, Piscataway, NJ 08854-8019, U.S.A.}
\email{vasconce@math.rutgers.edu}

\author{Rafael H. Villarreal}
\address{Departamento de
Matem\'aticas\\
Centro de Investigaci\'on y de Estudios
Avanzados del
IPN\\
Apartado Postal
14--740,
07000 Mexico City, D.F.}
\email{vila@math.cinvestav.mx}

\subjclass[2010]{Primary 13B22; Secondary 13A30, 13H10.}

\keywords{$\m$-full ideal, normal ideal, Newton polytope, Hilbert function,
     Rees algebra,  Jacobian dual, Cohen--Macaulay ring.}

\begin{abstract}
\noindent
In dimension two,
we  study complete monomial ideals  combinatorially,
their Rees algebras and develop  effective means to
find their defining equations.
\end{abstract}

\maketitle

\section{Introduction}
\noindent
The study of complete ideals in the polynomial ring $k[x,y]$ is a classical subject started by
Zariski in \cite{Z} (see also \cite[Appendix 5]{ZS}) and subsequently developed by various other authors  (\cite{Hu}, \cite{Mon}, \cite{L},  \cite{Rees}).

The special case of monomial ideals is enhanced by the use of combinatorics, specially those parts
related to convex hull techniques.
It is somewhat surprising that only more recently, this facet took off accordingly.
Thus, in  \cite{crispin-quinonez-thesis} and \cite{crispin-quinonez} Qui\~nonez studied the normality of
monomials ideals in $k[x,y]$ and established a criterion in terms of certain partial blocks and
associated sequences of rational numbers.

In the present work, the overall goal is to study normal monomial ideals in $k[x,y]$
 in a  landscape governed by Zariski's theory of complete ideals and the structures and
algorithms associated to Newton polygons.
A common root between  Qui\~nonez' approach and ours is the emphasis on the exponents of the
monomials that generate the given ideal written in lexicographic order with $x>y$, thus affording
a slightly different angle from the one in some of the previous classical approach.

A difference between our results and Qui\~nonez' lies in that we state necessary {\em or} sufficient conditions for
normality directly in terms of the stair sequences of the monomial exponents by means of
certain inequalities.
Since each of these criteria is stated by means of a finite set of numerical inequalities, it is doubtful whether
one can group them together in order to obtain a full characterization of normality
(this point is addressed in detail in Question~\ref{finiteness}).

Other points of contrast are our use of polyhedra theory (such as Pick's formula) and a strengthening of
the relationship between normality and $\m$-fullness -- the latter a concept introduced by Rees and
developed in \cite{HuSw} and \cite{JWat}.
Thus, the preponderance of our algebraic results are derived from the
properties of the polygon defined by the points in the plane whose coordinates
are the exponents of the monomials generating the ideal.
It benefits from the fact that a natural starting point is the direct
description of $\m$-full ideals and the
simplicity of their syzygies.

Unavoidably in such a narrowly defined class of ideals, coming from a slightly distinct view point we recover some
of the results of Qui\~nonez. In such cases, we explain the relationship between the two.

\medskip

Let $\RR = k[x,y]$, $\m = (x,y)$, and $I$ be a monomial ideal. When
needed in our references to the literature,
 we assume $k$ infinite.
 Suppose that $I$ is $\m$-primary, minimally generated by $n$ elements,
  $\mu(I) = n$, but
$I \neq \m^{n-1}$. $I$ is minimally generated by $n$ monomials that are listed lexicographically,
$I=(x^{a_1}, x^{a_2}y^{b_{n-1}},\ldots,x^{a_i}y^{b_{n-i+1}}, \ldots,
x^{a_{n-1}}y^{b_2}, y^{b_1})$
with
\[
a_1>a_2> \cdots>a_{n-1}>a_n:=0, \quad
b_1>b_2> \cdots>b_{n-1}>b_n:=0,
\]
defining the set of points $P_i=(a_i, b_{n-i+1})$, $1\leq i\leq n$.

\medskip

Our first result describes how given a monomial ideal $I$ to find the
smallest $\m$-full monomial ideal $I'$ containing it (Proposition~\ref{mfullclosmonid}).
This works for any $\m$-primary monomial ideals in $k[x_1,\ldots,x_d]$.
Moreover, in the case where $d=2$ we characterize when $I$ is $\m$-full (Theorem~\ref{p-w-v}).
In a different direction we take up the normality question, by
conveying several necessary conditions or sufficient conditions for it to hold, expressed by
systems of linear inequalities $Q(P_1, \ldots, P_n)\leq 0$
(Proposition~\ref{apr14-13}, Theorem~\ref{dim2eff} and
Proposition~\ref{dim2eff2}).

\medskip
Our most comprehensive results are given in the equations of the Rees algebras  $\RR[It]$ of normal monomial
ideals. They are put together from two facts. On one hand, the
 algebras $\RR[It]$ being normal are Cohen-Macaulay by a theorem of Lipman--Teissier
  (\cite[Corollary 5.4]{LT}). On the other hand, the syzygies of $I$ are straightforward enough to permit getting the equations of $\RR[It]$ in one or two rounds of elimination. An
 effective  application
 of  a theorem of
  Morey--Ulrich (\cite[Theorem~1.2]{MU}) gives the case when one round
  of elimination suffices (Theorem~\ref{Expeq}).
  This is an approach that has also been exploited in \cite[Theorem 3.17]{Conca05} and \cite{Conca07} allied with
  a detailed examination of their Hilbert functions
  for a wider class of ideals.
 Here    aiming for less generality
  we get to the equations as  quickly and effectively as possible by introducing
a second elimination round to recover them all
(Theorem~\ref{Expeq2}).
Finally we recall  that while the Rees algebras of $\m$-full ideals are not always Cohen-Macaulay, it
will be so if its special fiber is Cohen-Macaulay (Theorem~\ref{CMviafiber}).

\section{
Criteria for $\m$-fullness and normality}

\subsection{Polyhedra}

Let $\RR=k[x_1,\ldots,x_d]$ be a polynomial ring over a field $k$,
with $d\geq 2$, and let $I$ be a zero-dimensional  ideal
of $\RR$ minimally generated by monomials $x^{v_1},\ldots,x^{v_q}$,
where $x^{v_j}:=x_1^{v_{1,j}}\cdots x_n^{v_{n,j}}$, for $j=1,\ldots,q$.
Consider the
rational convex polyhedron $\mathcal{Q}:=\mathbb{R}_{\geq 0}^d+{\rm
conv}(v_1,\ldots,v_q)$, where ${\rm conv}(v_1,\ldots,v_q)$ denotes
the convex hull of $v_1,\ldots,v_q$ in $\mathbb{R}^d$. The integral polytope
${\rm conv}(v_1,\ldots,v_q)$, denoted by $N(I)$, is called the {\it
Newton polytope\/} of $I$ and $\mathcal{Q}$ is called the {\it Newton
polyhedron\/} of $I$.

\begin{Remark}\rm $\mathcal{Q}\cap\mathbb{Z}^d=
(\mathbb{Q}_{\geq 0}^d+{\rm
conv}_\mathbb{Q}(v_1,\ldots,v_q))\cap \mathbb{Z}^d$. This follows
using that $\mathcal{Q}$ is a rational polyhedron, i.e., the vertices
of $\mathcal{Q}$ are in $\mathbb{Q}^d$.
\end{Remark}

As usual, we denote the {\it floor\/} and {\it ceiling\/} of a real
number  $r$ by $\lfloor{r}\rfloor$ and $\lceil{r}\rceil$,
respectively.  One can use these notions to give necessary and
sufficient conditions for the normality of $I$ as well as some
descriptions of the integral closures of the
powers of $I$ \cite{ainv,multical, normali, clutters} (see also
Proposition~\ref{may29-05} and Theorem~\ref{dim2eff} below).

\begin{Definition}\rm Let $A$ be the $d\times q$ integer matrix with
column
vectors $v_1,\ldots,v_q$. The system $x\geq 0; xA\geq\mathbf{1}$ of linear
inequalities is said to have the {\it integer rounding property\/}
if
$${\rm max}\{\langle y,{\mathbf 1}\rangle \vert\,
y\geq 0; Ay\leq w; y\in\mathbb{N}^q\}
=\lfloor{\rm max}\{\langle y,{\mathbf 1}\rangle \vert\, y\geq 0;
Ay\leq w\}\rfloor
$$
for each integer vector $w$ for which the right
hand side is finite. Here $\mathbf{1}=(1,\ldots,1)$ and
$\langle,\rangle$ denotes the usual inner product.
\end{Definition}

Systems with the integer rounding property have been widely studied
from the viewpoint of integer programming; see \cite[pp. 336--338]{Schr},
\cite[pp. 82--83]{Schr2}, and the references there.

\begin{Proposition}\label{may29-05}
\begin{enumerate}
\item
{\rm\cite[Proposition~1.1]{multical}}
$\overline{I^m}=(\{x^a\vert\,
a\in{m\mathcal{Q}}\cap\mathbb{Z}^d\})$ for $0\neq m\in\mathbb{N}$.
\item\label{ii}
{\rm \cite[Proposition~1.2]{multical}} $\overline{I}$ is generated
by all $x^a$ with $a\in(\mathcal{Q}+[0,1)^d)\cap\mathbb{N}^d$.
\item
{\rm \cite[Corollary 2.5]{poset}} $I$ is normal if and only if
the system $x\geq 0;\,xA\geq\mathbf{1}$ has the integer rounding
property.
\end{enumerate}
\end{Proposition}

\begin{Example}\rm If $I=(x^2,y^3)$, then
$N(I)\cap\mathbb{Z}^2=\{(2,0),(0,3)\}$ and $\overline{I}=I+(xy^2)$.
\end{Example}

If $I\subset k[x,y]$, the next result will be used to give necessary condition for $I$ to be
normal (see Proposition~\ref{apr14-13}(b)).

\begin{Proposition}{\rm(Pick's Formula, \cite[p.~248]{BG-book})}\label{pick-formula} If
$\mathcal{P}\subset\mathbb{R}^2$ is an integral polytope of dimension $2$, then
$$
{\rm area}(\mathcal{P})=|\mathbb{Z}^2\cap \mathcal{P}|-
\frac{|\mathbb{Z}^2\cap\partial \mathcal{P}|}{2}-1=|\mathbb{Z}^2\cap\mathcal{P}^{\rm
o}|+\frac{|\mathbb{Z}^2\cap\partial \mathcal{P}|}{2}-1\index{Pick's
Formula},
$$
where $\partial \mathcal{P}$ and $\mathcal{P}^{\rm o}$ are the boundary and
the interior of $\mathcal{P}$, respectively.
\end{Proposition}

\subsection{Full ideals}\index{full ideal}

Let $\RR = k[x,y]$
 and $\m = (x,y)$. An $\m$-primary ideal $I$ is said to be {\em $\m$-full} if
$\m I: a = I$ for some $a\in \m \setminus \m^2$ (see \cite[14.1.5]{HuSw}).
The element $a$ can be taken to be a linear
form not dividing the {\em content ideal} $c(I)$ of $I$ (see \cite[14.1.1, Proposition 14.1.7]{HuSw}).
If $I$ is a monomial ideal, $a$
can be taken to be a form  that is not a monomial since the content is monomial. We shall take $a=x+y$.

The fundamental characterization of $\m$-full ideals of two-dimensional regular local rings is the
following result due to Rees (\cite[Exercise~14.1]{HuSw}, \cite[Theorem 4]{JWat}):

\begin{Theorem}\label{Reestheorem} Let $I$ be an $\m$-primary ideal of a regular local ring $(\RR,\m)$ of dimension
two. Then $I$ is $\m$-full if and only if for all ideals $I\subset J$, $\mu(J)\leq \mu(I)$.
\end{Theorem}

\medskip

In analogy to the existence of the integral closure of an ideal, let us consider the question of its {\rm $\m$-full} closure
in the sense of  a unique minimal {\rm $\m$-full} ideal
$J$ containing $I$. Since the set of {\rm $\m$-full} ideals containing $I$ is non-empty and satisfies the minimal chain condition
there may exist, to the authors' knowledge,  more than one  minimal element, a situation that makes appointing one of them as {\em the closure} not appropriate.
 The situation
is clearer if  we consider only  the set of   monomial ideals.

\begin{Proposition}\label{mfullclosmonid}
Let $I$ be an $\m$-primary monomial ideal.
Then its {\rm $\m$-full} monomial closure $I^*$ exists and it is integral over $I$.
\end{Proposition}

\begin{proof}  For any ideal $L$ of a polynomial ring $k[x_1, \ldots, x_d]$, we
denote by $M(L)$ the ideal generated by all monomials that occur in the representation of the
elements of $L$. It is clear that $M(L)$ is defined by  the monomials that occur in any generating set
for $L$.
Note that if $J$ is a monomial ideal and $L\subset J$, then $M(L) \subset J$.
 \begin{itemize}
 \item Consider the set of all {\rm $\m$-full} monomial ideals that  contain the monomial ideal
 $I$.
For each such
{\rm $\m$-full} monomial ideal $L$, we have
\[ \m I: x+y \subset  \m L: x+y = L.\]
\item
If $\m I: x+y \neq I$, that is if $I$ is not {\rm $\m$-full}, note
that $ M(\m I: x+y)$ properly
contains $I$ but it is still contained in $L$.
In this case, set $I_1 = M(\m I: x+y)$
and apply the previous step to it. This process defines an increasing chain of
monomial ideals $I\subset I_1 \subset I_2 \subset \cdots$, contained in $L$ whose stable ideal
$I^*$ is {\rm $\m$-full}.
\item Considering that the integral closure
$\bar{I}$ of $I$ is monomial and $\m$-full we have $I^*\subset \bar{I}$.
\end{itemize}
\end{proof}

\begin{Example} {\rm Let $I = (x^3, y^5)$. Then $\m I:x+y = (x^3, x^2y^3-xy^4, y^5)$. Thus
$I_1 = (x^3, x^2y^3, xy^4, y^5)$ and $\m I_1:x+y = I_1$. Thus $I^* = I_1$.
}\end{Example}

Let us cast Theorem~\ref{Reestheorem} for monomial ideals of $k[x,y]$ into an effective  form for later usage.

\begin{Theorem}\label{p-w-v}
If $I$ is minimally generated by $n$ monomials that are listed lexicographically,
$I=(x^{a_1}, x^{a_2}y^{b_{n-1}},\ldots,x^{a_i}y^{b_{n-i+1}}, \ldots,
x^{a_{n-1}}y^{b_2}, y^{b_1})$,
then $I$ is $\m$-full if and only if there is $1\leq k\leq n$ such
that the following conditions hold
\begin{enumerate}
\item
$b_{n-i}-b_{n-i+1}=1$ for $1\leq i\leq k-1$,
\item
$k=n$ or $k<n$ and $b_{n-k}-b_{n-k+1}\geq 2$,
\item
$a_i-a_{i+1}=1$ for $k\leq i\leq n-1$.
\end{enumerate}
\end{Theorem}

\begin{proof}
We first show that $I$ is $\m$-full if and only if  $\mbox{\rm order}(I)= n-1$.
Thus, suppose  $\mbox{\rm order}(I)= n-1$, i.e., there is an element $x^{k}y^{n-1-k}\in I$.
Note that $I\subset J$ implies $\mbox{\rm order}(J)\leq \mbox{\rm order}(I)$. Since
$\mu(J) \leq \mbox{\rm order}(J)+1\leq n$, $I$ satisfies Theorem~\ref{Reestheorem}.

Conversely, if $I$ is $\m$-full, $\mbox{\rm order}(I)\leq n-1$, as otherwise $I \subset (x,y)^n$, which has
$n+1$ minimal generators, which would violate Theorem~\ref{Reestheorem}.

\medskip

Now, granted  $\mbox{\rm order}(I)= n-1$, suppose an occurrence of a monomial of degree $n-1$ is  $x^k y^{n-1-k}$. This means that there are at
most $n-1-k$ elements prior to $x^ky^{n-1-k}$ and at most $k$ elements after.
This gives
\[I=(x^{a_1}, x^{a_2}y,\ldots,x^{a_{k-1}}y^{{n-2-k}}, x^ky^{n-1-k}, x^{k-1}y^{b_{n-k}},  \ldots, xy^{b_2},
 y^{b_1}).\]
By choosing $k$ as small as possible we achieve all three conditions.

Conversely, it is clear that the set of the three stated conditions implies  $\mbox{\rm order}(I)= n-1$.
\end{proof}

\begin{Corollary}\label{corq}
Let $I$ be an ideal minimally generated by $n$
monomials that are listed lexicographically,
$I=(x^{a_1}, x^{a_2}y^{b_{n-1}},
\ldots,x^{a_i}y^{b_{n-i+1}},\ldots, x^{a_{n-1}}y^{b_2}, y^{b_1})$.
Suppose that $I$ is normal.
\begin{enumerate}
\item
For every $i$, either  $a_i-a_{i+1}= 1$ or $b_{n-i}-b_{n-i+1}= 1$.
\item\label{q2}
If $b_{n-i}-b_{n-i+1}>1$ for some $i$, then $a_{i}-a_{i+1}=1$ and $a_{i+1}-a_{i+2}=1$.
\end{enumerate}
\end{Corollary}

\begin{proof}
Both claims follow readily from Theorem~\ref{p-w-v} because
complete monomial ideals of $k[x,y]$ are $\m$-full
\cite[Theorem~14.1.8]{HuSw}.
\end{proof}

This result has been already observed in \cite[p.~369]{crispin-quinonez}.
In the same work, the following terminology was introduced for zero-dimensional monomial ideals
$I\subset k[x,y]$, whose generators are ordered as above:
$I$ is called $x$-{\it tight} (resp. $y$-{\it tight}) if
$a_i-a_{i+1}=1$ for all $i$ (resp. $b_i-b_{i+1}=1$ for all $i$).

Putting together the previous result and these notions, we have:

\begin{Corollary}\label{mfull_tight}
If $I$ as above is $\m$-full then it is the product of an $x$-tight ideal and a $y$-tight ideal.
\end{Corollary}
\begin{proof}
It follows from Theorem~\ref{p-w-v} and \cite[Proposition~2.2]{crispin-quinonez}.
\end{proof}

 \subsection{Normality criteria}
 We now proceed to establish separate necessary or sufficient conditions for normality
 in terms of the associated monomial exponents.

First necessary conditions:

\begin{Proposition}\label{apr14-13}
Let $I$ be an ideal minimally generated by $n$
monomials that are listed lexicographically,
$I=(x^{a_1}, x^{a_2}y^{b_{n-1}},
\ldots,x^{a_i}y^{b_{n-i+1}},\ldots, x^{a_{n-1}}y^{b_2}, y^{b_1})$
and let $P_i=(a_i,b_{n-i+1})$, $P_{i+1}=(a_{i+1},b_{n-i})$,
$P_{i+2}=(a_{i+2},b_{n-i-1})$ be three consecutive points corresponding to the
exponents of the defining monomials of $I$.
The following hold:
\begin{enumerate}
\item[(a)]
If $I=\overline{I}$ and $b_{n-i-1}-b_{n-i}=1$,
$b_{n-i}-b_{n-i+1}=1$, then $a_{i+1}\leq
\lceil\frac{a_{i}+a_{i+2}}{2}\rceil$.
\item[(b)]
If $I=\overline{I}$ and $a_{i}-a_{i+1}=1$,
$a_{i+1}-a_{i+2}=1$, then $b_{n-i}\leq
\lceil\frac{b_{n-i-1}+b_{n-i+1}}{2}\rceil$.
\end{enumerate}
\end{Proposition}

\begin{proof}
(a): First we assume that $P_{i+1}\in{\rm conv}(P_i,P_{i+2})$. Then,
we can write $P_{i+1}=\lambda_1P_i+\lambda_2P_{i+2}$, where
$\lambda_i>0$, $i=1,2$ and $\lambda_1+\lambda_2=1$. It is not hard to
see that $\lambda_i=1/2$ for $i=1,2$. Thus, one has
$a_{i+1}=\frac{a_{i}+a_{i+2}}{2}=\lceil\frac{a_{i}+a_{i+2}}{2}\rceil$.

We may now assume that $P_{i+1}\notin{\rm conv}(P_i,P_{i+2})$. We
proceed by contradiction
assuming that $a_{i+1}>\lceil\frac{a_{i}+a_{i+2}}{2}\rceil$, that is
$a_{i+1}-\frac{a_{i}+a_{i+2}}{2}\geq 1$. Consider the convex polytope
$\mathcal{P}$ whose vertices are $P_i,P_{i+1},P_{i+2}$. We claim that
$\partial\mathcal{P}\cap\mathbb{Z}^2=\{P_i,P_{i+1},P_{i+2}\}$. Clearly
${\rm conv}(P_i,P_{i+1})^{\rm o}\cap\mathbb{Z}^2=\emptyset$ and
${\rm conv}(P_{i+1},P_{i+2})^{\rm o}\cap\mathbb{Z}^2=\emptyset$
because $b_{n-i-1}-b_{n-i}=1$ and $b_{n-i}-b_{n-i+1}=1$. We claim
 that also ${\rm conv}(P_i,P_{i+2})^{\rm o}\cap\mathbb{Z}^2=\emptyset$.
Indeed, if this set is non-empty, pick an integral point $(c_1,c_2)$ in ${\rm
conv}(P_i,P_{i+2})^{\rm o}$. By Proposition~\ref{may29-05}, the
monomial $x^{c_1}y^{c_2}$ is in $\overline{I}=I$. Then we can write
\begin{eqnarray}
c_1&=& ta_{i+2}+(1-t)a_i=\epsilon_1+a_j\label{eq1},\\
c_2&=& tb_{n-i-1}+(1-t)b_{n-i+1}=\epsilon_2+b_{n-j+1}\label{eq2},
\end{eqnarray}
for some $j$, where $\epsilon_1,\epsilon_2$ are in $\mathbb{N}$ and
$0<t<1$. From Eq.~(\ref{eq1}), we get $a_i>c_1=\epsilon_1+a_j$.
Thus $i<j$. From Eq.~(\ref{eq2}), we get
\begin{eqnarray*}
c_2&=&t(b_{n-i-1}-b_{n-i+1})+b_{n-i+1}=2t+b_{n-i+1}=\epsilon_2+b_{n-j+1}.
\end{eqnarray*}
Hence $2+b_{n-i+1}>\epsilon_2+b_{n-j+1}$. If $\epsilon_2 \geq 1$,
then $b_{n-i+1}-b_{n-j+1}\geq 0$ and consequently $i\geq j$, a
contradiction. Hence, $\epsilon_2=0$, $j=i+1$ and $t=1/2$. Therefore
from Eq.~(\ref{eq1}), we obtain
$\epsilon_1=\frac{a_{i}+a_{i+2}}{2}-a_{i+1}\geq 0$, a contradiction.
This completes the proof of the claim. As a consequence, using Pick's
formula (Proposition~\ref{pick-formula}), one has
\begin{equation}\label{eq3}
{\rm area}(\mathcal{P})=|\mathcal{P}^{\rm
o}\cap\mathbb{Z}^2|+\frac{1}{2}.
\end{equation}
The equation of the line passing through $P_i$ and $P_{i+2}$ is
$$
x_1(b_{n-i-1}-b_{n-i+1})+x_2(a_i-a_{i+2})=a_i(b_{n-i-1}-b_{n-i+1})+(a_i-a_{i+2})b_{n-i+1}.
$$
Since $a_{i+1}-\frac{a_{i}+a_{i+2}}{2}\geq 1$, the point $P_{i+1}$
lies above this line. It follows readily that the area of
$\mathcal{P}$ is given by
$$
{\rm area}(\mathcal{P})=a_{i+1}-\frac{a_{i}+a_{i+2}}{2}\geq 1.
$$
Hence, by Eq.~(\ref{eq3}), $\mathcal{P}^{\rm
o}\cap\mathbb{Z}^2\neq\emptyset$.
Pick an integral point $(c_1,c_2)$ in $\mathcal{P}^{\rm o}$. By Proposition~\ref{may29-05}, the
monomial $x^{c_1}y^{c_2}$ is in $\overline{I}=I$. Then we can write
\begin{eqnarray}
c_1&=&
\lambda_1a_i+\lambda_2a_{i+1}+\lambda_3a_{i+2}=\epsilon_1+a_j\label{eq1-bis},\\
c_2&=&
\lambda_1b_{n-i+1}+\lambda_2b_{n-i}+\lambda_3b_{n-i-1}=\epsilon_2+b_{n-j+1},\label{eq2-bis}
\end{eqnarray}
for some $j$, where $\epsilon_1,\epsilon_2$ are in $[0,1)$,
$\lambda_i>0$ for $i=1,2,3$ and $\lambda_1+\lambda_2+\lambda_3=1$.
From Eqs.~(\ref{eq1-bis}) and (\ref{eq2-bis}),
we get $a_i>c_1=\epsilon_1+a_j$ and $b_{n-i-1}>c_2=\epsilon_2+b_{n-j+1}$.
Thus $i<j$ and $-2<i-j$, i.e., $j=i+1$. Therefore we can rewrite
Eq.~(\ref{eq2-bis}) as
\begin{eqnarray*}
c_2
&=&\lambda_1(b_{n-i}-1)+\lambda_2b_{n-i}+\lambda_3(b_{n-i}+1)\\
&=&b_{n-i}-\lambda_1+\lambda_3
=\epsilon_2+b_{n-i}.
\end{eqnarray*}
As a consequence $-\lambda_1+\lambda_3=\epsilon_2\geq 0$. Hence
$\epsilon_2$ must be zero because $\epsilon_2<\lambda_3<1$. Then from
From Eq.~(\ref{eq1-bis}), we get
\begin{eqnarray*}
c_1&=&\lambda_1(a_i+a_{i+2})+\lambda_2a_{i+1}=\lambda_1(a_i+a_{i+2})+
(1-2\lambda_1)a_{i+1}
=\epsilon_1+a_{i+1}.
\end{eqnarray*}
Thus $\lambda_1(a_i+a_{i+2}-2a_{i+1})=\epsilon_1\geq 0$, and hence
$a_i+a_{i+2}-2a_{i+1}\geq 0$, a contradiction.

(b): Notice that the ideal obtained from $I$ by permuting $x$ and $y$
is also normal. Thus this part follows from (a).
\end{proof}

\medskip

Putting together Corollary~\ref{corq} and Proposition~\ref{apr14-13} we obtain the following:

\begin{Theorem}\label{dim2eff}
Let  $I$  be minimally generated by $n$ monomials that are listed lexicographically,
$I=(x^{a_1}, x^{a_2}y^{b_{n-1}},\ldots,x^{a_i}y^{b_{n-i+1}}, \ldots, x^{a_{n-1}}y^{b_2}, y^{b_1})$.
If $I$ is normal, then there exists $k$, $1\leq k\leq n$, such that
\begin{enumerate}
\item
$a_{n-1}=1$, $a_{n-2}=2$, \ldots, $a_{k}=n-k$,
\item
$b_{n-1}=1$, $b_{n-2}=2$, \ldots, $b_{n-k+1}=k-1$,
\item
$b_{2}\leq \lceil\frac{b_1+b_3}{2}\rceil$, $b_{3}\leq \lceil\frac{b_2+b_4}{2}\rceil$, \ldots, $b_{n-k}\leq \lceil\frac{b_{n-k-1}+b_{n-k+1}}{2}\rceil$,
\item
$a_{2}\leq \lceil\frac{a_1+a_3}{2}\rceil$, $a_{3}\leq \lceil\frac{a_2+a_4}{2}\rceil$, \ldots, $a_{k-1}\leq \lceil\frac{a_{k-2}+a_{k}}{2}\rceil$.
\end{enumerate}
\end{Theorem}

\begin{proof} There is $1\leq k\leq n$ such that
$b_{n-1}=1$, $b_{n-2}=2$,\ldots, $b_{n-k+1}=k-1$ and
$b_{n-k}-b_{n-k+1}\geq 2$. Then, using Corollary~\ref{corq}(\ref{q2}), it is
seen that $a_i-a_{i+1}=1$ for $i\geq k$. Hence (1) and (2) hold. Parts
 (3) and (4) follow from Proposition~\ref{apr14-13}.
\end{proof}

\begin{Example}\label{counter-example}\rm The ideal
$I=(x^3,\,x^2y^8,\, xy^{15},\, y^{21})$ is not normal
 (but it is $\m$-full)
and satisfies the conditions of
Theorem~\ref{dim2eff}. The integral closure of $I$ is
$\overline{I}=(x^3,\,x^2y^7,\, xy^{14},\, y^{21})$.
\end{Example}

We next state sufficient conditions of similar nature for normality.

\begin{Proposition} \label{dim2eff2} Let $I\subset K[x,y]$ be an ideal minimally generated by $n$
monomials that are listed lexicographically,
$I=(x^{a_1}, x^{a_2}y^{b_{n-1}},\ldots,x^{a_i}y^{b_{n-i+1}}, \ldots,
x^{a_{n-1}}y^{b_2}, y^{b_1})$.
If $a_i-a_{i+1}=1$ for $i=1,\ldots,n-1$ and $2b_{n-i}\leq b_{n-i-1}+b_{n-i+1}$
for all $i$, then $I$ is normal.
\end{Proposition}

\begin{proof} Notice that $a_i=n-i$ for $i=1,\ldots,n$. Let
$x^{c_1}y^{c_2}$ be a minimal monomial generator of $\overline{I}$. By
Proposition~\ref{may29-05}(\ref{ii}) we can write
\begin{eqnarray}
c_1&=&
\lambda_1(n-1)+
\lambda_2(n-2)+\cdots+\lambda_i(n-i)+\cdots+\lambda_{n-2}(2)+\lambda_{n-1}(1)+\epsilon_1,
\label{may23-13-2}\\
c_2&=&
\lambda_2b_{n-1}+\cdots+\lambda_{i-1}b_{n-i+2}+\lambda_ib_{n-i+1}+
\lambda_{i+1}b_{n-i}+\cdots+\lambda_{n-1}b_2+\lambda_nb_1+\epsilon_2,\label{may22-13}
\end{eqnarray}
where $\epsilon_1,\epsilon_2$ are in $[0,1)$, $\lambda_i\geq 0$ for
all $i$ and $\sum_{i=1}^n\lambda_i=1$. Hence $0\leq c_1<n$. As $c_1$ is an
integer, one has $0\leq c_1\leq n-1$. Thus, $c_1=a_i=n-i$ for
some $1\leq i\leq n$. To show that $x^{c_1}y^{c_2}$ is in $I$ it
suffices to show that $x^{c_1}y^{c_2}$ is a multiple of
$x^{a_i}y^{b_{n-i+1}}$. The proof reduces to showing that $c_2\geq
b_{n-i+1}$.  Thus, by Eq.~(\ref{may22-13}), we need only show the
following inequality
\begin{eqnarray}
& &\lambda_2b_{n-1}+\cdots+\lambda_{i-1}b_{n-i+2}+
\lambda_{i+1}b_{n-i}+\cdots+\lambda_{n-1}b_2+\lambda_nb_1\geq
(1-\lambda_i)b_{n-i+1}.\label{may23-13}
\end{eqnarray}
Using $1-\lambda_i=\sum_{j\neq i}\lambda_j$, it follows that this inequality is
equivalent to
\begin{eqnarray}
& &\lambda_{i+1}(b_{n-i}-b_{n-i+1})+\lambda_{i+2}(b_{n-i-1}-b_{n-i+1})+\cdots+\lambda_{n}(b_1-b_{n-i+1})
\geq\label{may23-13-1}\\
& &\ \ \  \ \ \ \ \ \ \ \  \ \ \ \ \
\lambda_1b_{n-i+1}+\lambda_2(b_{n-i+1}-b_{n-1})+\cdots
+\lambda_{i-1}(b_{n-i+1}-b_{n-i+2}).\nonumber
\end{eqnarray}
From Eq.~(\ref{may23-13-2}) and using the equality
$n-i=(n-i)\sum_{j=1}^n\lambda_j$ one has
\begin{equation}\label{may23-13-3}
\lambda_{i+1}(1)\geq\lambda_1(i-1)+\cdots+\lambda_{i-1}(1)-
[\lambda_{i+2}(2)+\cdots+\lambda_{n-1}(n-i-1)+\lambda_n(n-i)].
\end{equation}
Hence to show Eq.~(\ref{may23-13-1}) it suffices to prove the
following inequality
\begin{eqnarray}
&
&\lambda_{1}[(i-1)(b_{n-i}-b_{n-i+1})-b_{n-i+1}]+\cdots+\lambda_{i-1}[(b_{n-i}-b_{n-i+1})-
(b_{n-i+1}-b_{n-i+2})]+\label{may23-13-4}\\
& &\  \ \ \ \ \
\lambda_{i+2}[(b_{n-i-1}-b_{n-i+1})-2]+\cdots+\lambda_n[(b_1-b_{n-i+1})-(n-i)]\geq
0.\nonumber
\end{eqnarray}
To complete the proof notice that this inequality holds because all
coefficients of $\lambda_1,\ldots,\lambda_n$ are non-negative.
\end{proof}

\begin{Example}\rm Let $I$ be the ideal of $\mathbb{Q}[x,y]$
generated by $x^2,xy^2,y^3$. This ideal is normal,
satisfies $a_i-a_{i+1}=1$ for $i=1,2$ but $2b_{2}\not\leq
b_{1}+b_{3}$, where $b_1=3$, $b_2=2$ and $b_3=0$.
\end{Example}

\begin{Corollary}\label{suffcond4normality}
Let $I$ be minimally generated by $n$ monomials that are listed lexicographically,
$I=(x^{a_1}, x^{a_2}y^{b_{n-1}},\ldots,x^{a_i}y^{b_{n-i+1}}, \ldots,
x^{a_{n-1}}y^{b_2}, y^{b_1})$. Assume that $I$ is $\m$-full and let $k$ be the integer obtained in Theorem \ref{p-w-v}, $1\leq k\leq n$.
If
\begin{enumerate}
\item
$2b_{2}\leq b_1+b_3$, $2b_{3}\leq b_2+b_4$, \ldots, $2b_{n-k}\leq b_{n-k-1}+b_{n-k+1}$,
\item
$2a_{2}\leq a_1+a_3$, $2a_{3}\leq a_2+a_4$, \ldots, $2a_{k-1}\leq a_{k-2}+a_{k}$
\end{enumerate}
then $I$ is normal.
\end{Corollary}

\begin{proof}
As observed in Corollary~\ref{mfull_tight}, an $\m$-full ideal $I$ is the product of an $x$-tight ideal $X$
and a $y$-tight ideal $Y$.
Moreover, by \cite[Proposition 2.6]{crispin-quinonez}, the product of an $x$-tight ideal and a $y$-tight ideal
is integrally closed if and only if both ideals are integrally closed. One can apply Proposition \ref{dim2eff2}
to $X$, and the similar
result holding for $y$-tight ideals to $Y$ to get the required result.
\end{proof}

\begin{Remark}{\rm
For an $\m$-full ideal, normality is a condition in between the set of conditions (3)-(4) in Theorem~\ref{dim2eff},
and the set of conditions (1)-(2) in Corollary \ref{suffcond4normality}.
}\end{Remark}

\subsubsection*{Related questions}

\begin{Question}\label{finiteness}\rm (Finiteness Question)
Each of the necessary and sufficient conditions of normality above is cast in the form of
a system $Q$ of  linear inequalities on the coordinates of the points $P_i$. It is not likely that
a full set of conditions can be expressed by a finite set $Q_1, \ldots, Q_m$ of inequalities.
More precisely for each type of such inequality $Q$ denote by $M(Q)$ the set of all monomial ideals that
satisfies $Q$.
For instance, for the normal ideals $I$ lying in the {\it variety}  $M(Q_i)$, then for all
pairs of integers $a,b\geq 1$, the ideal $(x^a, y)(x,y^b)I$ is also normal, by Zariski's theorem,
so it must belong to one of the other
 varieties $M(Q_j)$.
\end{Question}

 \begin{Question}\rm (Realization Question)
  Let $I$ be an $\m$-full ideal minimally generated by
$n$ elements and $\bar{I}$ its integral closure. Since $\bar{I}$ is also minimally generated by $n$ elements,
there is at least one map $\varphi$ between the set of points
$\{P_1, \ldots, P_n\}$ of $I$ and
$\{P_1', \ldots, P_n'\}$ of $\bar{I}$ given by $P_i = P_j' + R_{ij}$, for each $i$ and some $j$.
 Note that $R_{1j} = R_{nj} = (0,0)$.
We ask what is the nature of such maps? Is there more than one such mapping? A positive answer would help
in predicting the integral closure of a monomial ideal by first determining its $\m$-full closure.
\end{Question}

\section{Rees algebras}

 Let $I$ be a monomial ideal of $\RR=k[x,y]$. We now study the Rees algebras $\RR[It]$ emphasizing
 when they are Cohen-Macaulay and obtaining their defining equations.

\subsection{Syzygies}
We have the following facts about their syzygies.

\medskip

 {\bf [Matrix of syzygies]}:
Let $I$ be an ideal minimally generated by $n$
monomials that are listed lexicographically,
$I=(x^{a_1}, x^{a_2}y^{b_{n-1}}, \ldots,x^{a_i}y^{b_{n-i+1}},\ldots, x^{a_{n-1}}y^{b_2}, y^{b_1})$.
Among the Taylor syzygies, a subset of  $n-1$ ``consecutive'' ones minimally generate,
giving rise to the $n \times (n-1)$ syzygy matrix
\[ \varphi = \left[ \begin{array}{rrcrr}
y^{b_{n-1}} & 0 & \cdots & 0 & 0 \\
-x^{a_1-a_2} & y^{b_{n-2}-b_{n-1}} &\cdots  & 0 & 0\\
0 & -x^{a_2-a_3} & \cdots & 0 & 0  \\
\vdots & \vdots & \cdots & \vdots & \vdots \\
0 & 0 & \cdots & y^{b_{2}-b_{3}} & 0\\
0 & 0 & \cdots & -x^{a_{n-2}-a_{n-1}} &
 y^{b_{1}-b_{2}}   \\
0 & 0 & \cdots  &0 & -x^{a_{n-1}}  \\
\end{array}
\right].
\]

Note that $\varphi$ is monomial (this is not typical of monomial
ideals in higher dimension, it is even an
issue of which Cohen-Macaulay monomial ideals of codimension two have
a minimal presentation with monomial
entries). In particular we have:

\begin{Proposition} \label{syzygies}
Let $I$ be a codimension two monomial ideal of $\RR = k[x,y]$. The content ideal
of the syzygies of $I$ is $I_1(\varphi) = (x^r,y^s)$.
\end{Proposition}

A similar assertion holds  for an $\m$-primary $\m$-full ideal of a two dimensional regular local ring $(\RR,\m)$
 (of infinite residue field): $I_1(\varphi) = (x,f)$, $x\in \m \setminus \m^2$.

\subsection{Equations of the Rees algebra} Let $I$ be an ideal of $\RR$ minimally generated
by $n$ monomials.
 Let $\BB=\RR[\TT_1, \ldots, \TT_n] \rar \RR[It]$ be an $R$-algebra presentation of the Rees
algebra of $I$, and set $Q$ to be the kernel.  $Q$ is a graded prime ideal in the standard $R$-grading of $\BB$, $Q = Q_1 + Q_2 + \cdots$. With
the  syzygies defining $Q_1$, we focus on $Q_2$.

\medskip
\begin{itemize}

\item  {\bf [Elimination]}: Write the set $Q_1$ of syzygies of $I$ as
\[ Q_1=[\TT_1, \ldots, \TT_n] \cdot \varphi = \TT \cdot \varphi,\]
which we rewrite as
\[ \TT \cdot \varphi = I_1(\varphi)\cdot \BB(\varphi),\]
($\BB(\varphi)$ is called the Jacobian dual of $\varphi$)
where $I_1(\varphi)$ is represented as $[x^r,y^s]$.  By elimination
\[ I_2(\BB(\varphi))\subset Q_2\subset Q.\]

\medskip

\item {\bf [Expected equations]:} $I$ is said to have the expected equations if $Q=
(Q_1, I_2(\BB(\varphi)) )$.
Our setting is now ready for several applications of \cite{MU}.
See also \cite[Theorem 3.17]{Conca05} where a similar development takes place.

\end{itemize}

We will make use of the following criterium of Cohen-Macaulayness of Rees algebras.

\begin{Proposition}[{\bf Cohen-Macaulay Test}] \label{ReesCMest}
 Let $(\RR,\m)$ be a Cohen-Macaulay local ring of dimension two and
infinite  residue field. If $I$ is an $\m$-primary ideal, then $\RR[It]$ is Cohen-Macaulay if and only if
the reduction number of $I$ is at most $1$.
\end{Proposition}

\begin{proof}
The forward assertion is a consequence of the Goto-Shimoda theorem (\cite{GS82}) for rings of dimension two.
For the converse, if $J$ is a minimal reduction and $I^2 = JI$, $I\RR[It] = I\RR[Jt]$, from which it follows
that $I\RR[It]$ is a maximal Cohen-Macaulay module over $\RR[It]$. The Cohen-Macaulayness of $\RR[It]$ follows
from this (see \cite[p. 102]{icbook}).
\end{proof}

\begin{Theorem} \label{Expeq}
  The  Rees algebra of a complete ideal $I$ of $\RR$ is always Cohen-Macaulay. In particular,
$I$ has reduction number $\leq 1$.
 $I$ has the expected equations if and only if
 \[ I_{n-2}(\varphi) = I_1(\varphi)^{n-2}.\]
\end{Theorem}

The Cohen-Macaulayness is the result of Lipman--Teissier (\cite[Corollary 5.4]{LT}).
The last assertion follows from \cite[Theorem 1.2]{MU} and the observations above on the syzygies
of $I$.

\begin{Example}{\rm
 Suppose $I = (x, y^{b_1})\cdots (x, y^{b_{n-1}})$, $b_1 \leq \cdots \leq b_{n-1}$, $n\geq 4$.
Consider its matrix $\varphi$ of syzygies. Inspection gives: $I_1(\varphi) = (x, y^{b_1})$, while the required  equality
\[ I_{n-2}(\varphi) = I_1(\varphi)^{n-2},\] that is
 \[ (x^{n-2}, x^{n-3}y^{b_1}, x^{n-4}y^{b_1+b_2}, \ldots, y^{b_1+b_2 + \cdots + b_{n-1}})= (x, y^{b_1})^{n-2},\]
 means \[b_1  = b_2 =  \cdots = b_{n-1}\] and therefore $I = (x, y^{b_1})^{n-1}$.
}\end{Example}

 These observations mean that at least
 among standard ideals those with the expected equations are rare.
 If $I$ is normal but does not have the expected equations, where are the missing equations? A guess [to be proved below]  is that
they are quadratic, missing from $I_{2}(\BB(\varphi))$.
 Note that if $I$ has the expected equations,
 \[ Q = I_1(\varphi)\BB(\varphi): I_1(\varphi).\]
  Since the right-hand side is always contained in $Q$,  we now discuss the case of  equality.
 If $I$ has the expected equations, $K=\TT\cdot \varphi+I_2(\BB(\varphi))$ is a prime ideal of
$\RR[\TT]$
of height $n-1$.
We can rewrite $(K, (x,y))$ (an ideal of height $n$) as
\begin{eqnarray} \label{expeq}
 (K + (x,y))=(L,(x,y)),
\end{eqnarray}
where
$L$
is the ideal of $k[\TT]$ of the maximal minors of the $2\times (n-1)$ matrix $\BB_0(\varphi)$ obtained from $\BB(\varphi)$ by
reduction mod $(x,y)$. By the Eagon-Northcott formula,
 \[\height L\leq (n-1)-2+1 = n-2.\] The equality $\height L = n-2$ now follows from (\ref{expeq}).
 Thus $L$ is Cohen-Macaulay. We note that with
this we have that the regularity of $k[\TT]/L$ is $1$ since $\BB_0(\varphi)$ is a matrix with linear entries.

\begin{Theorem} Let $I$ be a  monomial ideal such that $\RR[It]$ is Cohen-Macaulay. Let $\varphi$ be the matrix of
syzygies of $I$ and $\BB_0(\varphi)$ the matrix of linear forms of $k[\TT]$ defined above.
 The following conditions are equivalent:
\begin{enumerate}
\item $I$ has the expected equations;
\item $\height I_2(\BB_0(\varphi))= n-2$.
\end{enumerate}

\end{Theorem}

\begin{proof} It suffices to show that (2) implies (1). We will prove this by showing
that $K = (\TT\cdot \varphi, I_2(\BB(\varphi))$ is a prime ideal. Since
$\height (K, (x,y)) = n$, $\height K \geq n-2$. Let $P$ be a minimal prime of
$K$ of height $n-2$. $(x,y) \not \subset P$. Let $z\in (x,y) \setminus P$. Then the localization
$P_z$ is a minimal
prime of $K_z= (I_{n-1}(\varphi)\cdot \TT)_z$. But this is the defining ideal
of $\RR_z[t]$, so it has height $n-1$.

This shows that $K$ has height $n-1$.
$K$ is a specialization of a generic residual intersection of a complete intersection so it is
Cohen-Macaulay (\cite[Theorem 5.9]{HuUl88}).

To prove $Q=K$ it suffices to show that $K$ is prime (recall that $Q$ is a prime of height $n-1$).
As above we can pick $z\in (x,y)$ but avoiding every associated prime of $K$. But as we saw, $K_z$
is a prime ideal of height $n-1$. This is enough to show that $K $ is prime.
\end{proof}

\subsection{Full set of quadratic equations}
We shall describe where the quadratic relations of the Rees algebras $\RR[It]$ are located. In general, from a presentation
\[ 0 \rar Q \lar \BB = \RR[\TT_1, \ldots, \TT_n] \lar \RR[It] \rar 0,\]
$\BB/(Q_1)$ defines the symmetric algebra $\Sym(I)$ of $I$. We put
\[ 0 \rar \mathcal{A} = A_2 + A_3 + \cdots \lar \Sym(I) \lar \RR[It] \rar 0.\]
Here $A_2$ represents the effective quadratic relations of the Rees algebra $\RR[It]$, and we represent it as
\[0 \rar \delta(I) \lar S_2(I) \lar I^2 \rar 0.\]
For a discussion of $\delta(I)$,  see \cite{SV}. One of its properties gives $\delta(I)$ in the exact sequence
\[ 0\rar \delta(I) \lar \H_1(I) \lar (\RR/I)^n \lar I/I^2 \rar 0,\]
where $\H_1(I)$ is the first Koszul homology module on a set of $n$ generators of $I$. This says that $\delta(I)$
are the homology classes of the syzygies of $I$ with coefficients in $I$.

\begin{Theorem} \label{Expeq2} Let $(\RR,\m)$ be a two-dimensional regular local ring and $I$ an $\m$-primary ideal. If the Rees
algebra $\RR[It]$ is Cohen-Macaulay, then
\[ Q = (Q_1, Q_2)= (Q_1):I=(\TT\cdot \varphi):I.\]
\end{Theorem}

\begin{proof}
 Since  $\RR$ is Cohen--Macaulay, the reduction number $r(I)$ of $I$ satisfies $r(I) < \dim \RR=2$.
We now apply \cite[Theorem 1.2]{Trung87}: $\RR[It]$ is defined by linear and quadratic equations, $Q=(Q_1, Q_2)$
and \[ \ann(\delta(I))\cdot Q_2 \subset Q_1\BB_1.\]
Of course any nonzero ideal contained in $\ann(\delta(I))$ serves the purpose, in particular
$\ann(\H_1(I)) \supset I$ (actually there is equality).
This gives the assertion.
\end{proof}

Note that this does not require that $I_1(\varphi)$ be a complete intersection.

\begin{Example}{\rm  Let
 \[ I = (x, y)(x, y^3)(x,y^6) = (x^3, x^2y, xy^4, y^{10}).\]
$I$ is normal and
its matrix of syzygies is
\[
\varphi = \left[
\begin{array}{rrr}
y & 0 & 0 \\
-x & y^3 & 0 \\
0 & -x & y^6 \\
0 & 0  & -x \\
\end{array} \right].
\]

Note that $I_1(\varphi) = (x,y)$, but $I_2(\varphi) = (x^2, xy, y^4)\neq I_1(\varphi)^2$, so it does
not have the expected equations. We have
\[ I_2(\BB(\varphi)) = I_2\left(\left[\begin{array}{rrr}
-\TT_2 & -\TT_3 &  -\TT_4 \\
\TT_1 & y^2\TT_2 & y^5\TT_3 \\
\end{array} \right]\right),
\] which gives only two minimal generators for $Q_2$. An appeal to {\em Macaulay2} (\cite{Macaulay2}) gives the extra generator:
\begin{eqnarray*}
Q = (Q_1, Q_2) & = & \TT\cdot \varphi : I \\
& = & (\TT\cdot \varphi, I_2(\BB(\varphi)), \TT_2\TT_4-y^3\TT_3^2).
\end{eqnarray*}
}\end{Example}

An interesting question would ask about the arithmetical and homological properties of the Rees algebras of $\m$-full ideals.
Even for monomial ideals, these often fail to be Cohen-Macaulay, as the following example shows: $I = (x^{11}, x^8y, x^6y^2, x^5y^3, xy^4, y^{10})$, an $\m$-full ideal.
To show that $\RR[It]$  is {\em not} Cohen-Macaulay, by invoking {\em Macaulay2}, it  is enough
to verify that the special fiber $\mathcal{F}(I)$ of $I$ is {\em not} Cohen-Macaulay,
 according to the following criterion inspired
by
\cite[Corollary 2.11]{CGPU03}:

\begin{Theorem} \label{CMviafiber} Let $I$ be an $\m$-primary $\m$-full ideal. If the special fiber
 $F=\mathcal{F}(I)$ is Cohen-Macaulay then $\RR[It]$ is also Cohen-Macaulay.
 \end{Theorem}

 \begin{proof}
 Suppose $\mu(I)=n$ and let us determine the Hilbert function of $F$.
For every $j\geq 0$, \[ \mu(F_j) = j(n-1) + 1 ,\] since
$I^j$ is $\m$-full and contains an element of  order $j(n-1)$. It follows that the Hilbert series of $F$ is
\[ H_F(\ttt) = {\frac{1+ (n-2)\ttt }{(1-\ttt)^2}}.\]
This says that if $F$ is Cohen-Macaulay, as a  module over a Noether normalization $\AA=k[u,v]$, it is  $\AA$-free, with 1 generator of degree $0$ and $n-2$  generators of degrees
 \[ 1\leq d_1\leq d_2 \leq \cdots \leq d_{n-2}.\]
The Hilbert function forces $d_1 = \cdots = d_{n-2}= 1$.
 Therefore $I$ has reduction number
at most one.
\end{proof}

The same assertion holds for two-dimensional regular local rings of infinite residue field.

\end{document}